\documentclass[12pt, letterpaper]{amsart}
\usepackage{amsfonts,amsthm,mathtools}
\usepackage[all,cmtip]{xy}
\usepackage{fullpage}
\usepackage{amsmath}
\usepackage{amssymb}
\usepackage{enumerate}
\usepackage{tikz}
\usepackage[backref]{hyperref}
\usepackage{cleveref}
\usepackage{verbatim}
\usepackage{cancel}
\usepackage{float}
\usepackage{longtable}

\newtheorem{lem}{Lemma}[section]
\newtheorem{thm}[lem]{Theorem}

\newtheorem{prop}[lem]{Proposition}
\newtheorem{cor}[lem]{Corollary}
\newtheorem{conj}[lem]{Conjecture}

\theoremstyle{definition}
\newtheorem{remark}[lem]{Remark}
\newtheorem{definition}[lem]{Definition}

\DeclareMathAlphabet{\curly}{U}{rsfs}{m}{n}

\newcommand{\Q}{\mathbb{Q}}
\newcommand{\C}{\mathbb{C}}

\newcommand{\Z}{\mathbb{Z}}
\newcommand{\F}{\mathbb{F}}
\newcommand{\GL}{\operatorname{GL}}

\newcommand{\SL}{\operatorname{SL}}

\mathchardef\mhyphen="2D

\title{Tetragonal intermediate modular curves}

\author{\sc Petar Orli\'c}
\address{Petar Orli\'c \\
University of Zagreb\\  
Bijeni\v{c}ka Cesta 30 \\
10000 Zagreb\\
Croatia}
\email{petar.orlic@math.hr}

\begin{document}
\begin{abstract}
    For every group $\{\pm1\}\subseteq \Delta\subseteq (\Z/N\Z)^\times$, there exists an intermediate modular curve $X_\Delta(N)$. In this paper we determine all curves $X_\Delta(N)$ whose $\Q$-gonality is equal to $4$, all curves $X_\Delta(N)$ whose $\C$-gonality is equal to $4$, and all curves $X_\Delta(N)$ whose $\Q$-gonality is equal to $5$. We also determine the $\Q$-gonality of all curves $X_\Delta(N)$ for $N\leq 40$ and $\{\pm1\}\subsetneq \Delta \subsetneq (\Z/N\Z)^\times$.
\end{abstract}


\thanks{The author was supported by the project “Implementation of cutting-edge research and its application as part of the Scientific Center of Excellence for Quantum and Complex Systems, and Representations of Lie Algebras“, PK.1.1.02, European Union, European Regional Development Fund, and by the Croatian Science Foundation under the project no. IP-2022-10-5008.}

\maketitle

\section{Introduction}\label{introductionsection}

Modular curves $X(\Gamma)$ are a type of algebraic curves which can be constructed as quotients of the compactified upper half plane $\mathcal{H}^*$ with $\Gamma$, a congruence subgroup of the modular group $\textup{SL}_2(\Z)$. Most common modular curves are $X(N)$, $X_0(N)$, $X_1(N)$ which correspond to congruence subgroups

\begin{align*}
    \Gamma(N)&=\left\{ 
\begin{bmatrix}
a & b\\
c & d
\end{bmatrix}
\in \SL_2(\Z) : a,d\equiv 1\pmod{N}, \ b,c\equiv 0\pmod{N}
\right\},\\
\Gamma_0(N)&=\left\{ 
\begin{bmatrix}
a & b\\
c & d
\end{bmatrix}
\in \SL_2(\Z) : c\equiv 0\pmod{N} \right\},\\
\Gamma_1(N)&=\left\{ 
\begin{bmatrix}
a & b\\
c & d
\end{bmatrix}
\in \SL_2(\Z) : a,d\equiv 1\pmod{N}, \ c\equiv 0\pmod{N}
\right\}.
\end{align*}

For every group $\Delta\subseteq (\Z/N\Z)^\times$, there exists a modular curve $X_\Delta(N)$ defined over $\Q$. It corresponds to the congruence subgroup 
$$\Gamma_\Delta(N)=\left\{ 
\begin{bmatrix}
a & b\\
c & d
\end{bmatrix}
\in \SL_2(\Z) : (a\textup{ mod } N)\in\Delta, \ c\equiv 0\pmod{N}
\right\}.$$

Since $-I$ acts trivially on the upper half plane $\mathcal{H}^*$ (where $I\in\SL_2(\Z)$ is an identity matrix), the curves $X_\Delta(N)$ and $X_{\pm\Delta}(N)$ are isomorphic. Therefore, in this paper we will always assume that $-1\in\Delta$.

For every group $\{\pm1\}\subseteq \Delta\subseteq (\Z/N\Z)^\times$, there exists an intermediate modular curve $X_\Delta(N)$ lying between the curves $X_1(N)$ and $X_0(N)$. Notice that, when $\Delta=(\Z/N\Z)^\times$, the curve $X_\Delta(N)$ is actually the curve $X_0(N)$ and that, when $\Delta=\{\pm1\}$, the curve $X_\Delta(N)$ is the curve $X_{\pm1}(N)$ which is isomorphic to the curve $X_1(N)$. Moreover, if 
$$\{\pm1\}\subseteq \Delta_1\subseteq \Delta_2 \subseteq (\Z/N\Z)^\times,$$
then we have natural projections $$X_1(N)\to X_{\Delta_1}(N)\to X_{\Delta_2}(N)\to X_0(N)$$ defined over $\Q$.

It is also possible to define modular curves in another way. For every group $H\subseteq \GL_2(\Z/N\Z)$, there exists a modular curve $X_H$ defined over $\Q$. Using the same argument as before with the curves $X_\Delta(N)$, we may assume that $-I\in H$ without loss of generality. If $H$ has full determinant (that is, if $\det H=(\Z/N\Z)^\times$), the curve $X_H$ is ensured to be geometrically irreducible.

Suppose that $H\subseteq \GL_2(\Z/N\Z)$ such that $-I\in H$ and that $H$ has full determinant. Then there is a congruence subgroup $\Gamma$ such that the curves $X_H$ and $X(\Gamma)$ are isomorphic. It is defined as follows:
$$H_0:=\SL_2(\Z/N\Z)\cap H, \ \Gamma:=\{A\in\SL_2(\Z) : (A \textup{ mod } N) \in H_0\}.$$
It is not hard to check that $\Gamma(N)\subseteq \Gamma$, therefore $\Gamma$ is indeed a congruence subgroup of $\SL_2(\Z)$. Conversely, in the case of intermediate modular curves $X_\Delta(N)$, its isomorphic curve $X_H$ is defined as 
$$H:=\left\{  
\begin{bmatrix}
a & b\\
c & d
\end{bmatrix}
\in \GL_2(\Z/N\Z) : a\in\Delta, \ c=0
\right\}.$$
We can easily see that $-I\in H$ and that $H$ has full determinant.

Now we define the gonality of a curve. Let $C$ be a smooth projective curve over a field $k$. The $k$-gonality of $C$, denoted by $\textup{gon}_k C$, is the least degree of a non-constant $k$-rational morphism $f:C\to\mathbb{P}^1$.

The gonality of modular curves has been extensively studied. Zograf \cite{Zograf1987} gave a lower bound for the $\C$-gonality for any modular curve, linear in terms of the index of the corresponding congruence subgroup. Later, Abramovich \cite{abramovich} and Kim and Sarnak \cite[Appendix 2]{Kim2002} improved the constant in that bound.

Regarding the curve $X_0(N)$, Ogg \cite{Ogg74} determined all hyperelliptic, Bars \cite{Bars99} determined all bielliptic, Hasegawa and Shimura \cite{HasegawaShimura_trig} determined all trigonal curves $X_0(N)$ over $\C$ and $\Q$, and Jeon and Park \cite{JeonPark05} determined all tetragonal curves $X_0(N)$ over $\C$. Recently, Najman and Orlić \cite{NajmanOrlic22} determined all curves $X_0(N)$ with $\Q$-gonality equal to $4,5,$ or $6$. They also determined the $\Q$-gonality of all curves $X_0(N)$ for $N\leq144$ and for many of those curves also determined their $\C$-gonality.

Regarding the curve $X_1(N)$, Kenku and Momose \cite[p. 126]{KM88} determined all hyperelliptic, Jeon, Kim, and Schweizer \cite[Theorem 2.3]{JeonKimSchweizer04} determined all trigonal curves $X_1(N)$ over $\C$ and $\Q$, Jeon, Kim, and Park \cite[Theorem 2.6]{JeonKimPark06} determined all tetragonal curves $X_1(N)$ over $\C$ and $\Q$, and Derickx and van Hoeij \cite[Proposition 6]{derickxVH} determined all curves $X_1(N)$ with $\Q$-gonality equal to $d$ for $d=5,6,7,8$. They also determined the $\Q$-gonality of the curve $X_1(N)$ for $N\leq40$ and gave upper bounds on the $\Q$-gonality for $N\leq250$.

Now we move on to the intermediate modular curves $X_\Delta(N)$. Ishii and Momose \cite{IshiiMomose} determined (although with a slight error regarding the curve $X_{\Delta_1}(21)$) all hyperelliptic curves $X_\Delta(N)$ and Jeon and Kim \cite{Jeon2007} determined all trigonal curves $X_\Delta(N)$ over $\C$ and fixed this error. Jeon, Kim, and Schweizer \cite{JEON2020272} also determined all bielliptic curves $X_\Delta(N)$. Derickx and Najman \cite[Table 1]{DerickxNajman2023} determined the fields of definition of trigonal maps for genus $4$ curves $X_\Delta(N)$. This, together with the information about $\C$-trigonal curves $X_\Delta(N)$ from \cite{Jeon2007}, determines all curves $X_\Delta(N)$ which are trigonal over $\Q$.

\begin{thm}\cite[Table 1]{DerickxNajman2023}\label{CtrigonalQtetragonalthm}
    The genus $4$ intermediate modular curves $X_\Delta(N)$ that are trigonal over $\Q$ are
    \begin{align*}
        (N,\Delta)\in\{&(26,\{\pm1,\pm5\}),(26,\{\pm1,\pm3,\pm9\}),(28,\pm1,\pm3,\pm9),(28,\{\pm1,\pm13\}),\\
        &(29,\left<-1,4\right>),(37,\left<-1,8\right>),
        (37,\left<-1,4\right>),(50,\{\pm1,\pm9,\pm11,\pm19,\pm21\})\}.
    \end{align*}
    The only genus $4$ intermediate modular curve that is not trigonal over $\Q$ is $(N,\Delta)=(25,\{\pm1,\pm7\})$. For expository reasons, for larger groups $\Delta$ we give only their generators instead of all their elements.
\end{thm}

The next logical step is to determine all tetragonal curves $X_\Delta(N)$ over $\C$ and $\Q$. Also, since we know the $\Q$-gonality of curves $X_0(N)$ for $N\leq144$ and the $\Q$-gonality of curves $X_1(N)$ for $N\leq40$, we would like to obtain a similar result for intermediate curves $X_\Delta(N)$ for $N\leq40$. 

The main results of this paper are the following theorems.

\begin{thm}\label{Qgonalitythm}
    The $\Q$-gonalities of intermediate modular curves $X_\Delta(N)$ for all $N\leq40$ and $\{\pm1\}\subsetneq \Delta \subsetneq (\Z/N\Z)^\times$ are given in \Cref{tab:main1}.
\end{thm}

\begin{thm}\label{Qtetragonalthm}
    The intermediate modular curve $X_\Delta(N)$ is tetragonal over $\Q$ if and only if $N$ and $\Delta$ are listed in the following table. For expository reasons, we do not list all elements of larger groups $\Delta$. Instead, we give the generators and the number of elements of such groups $\Delta$.

        \begin{center}
\hspace*{-1cm}\begin{tabular}{|c|c||c|c|}
\hline
    $N$ & $\Delta$ & $N$ & $\Delta$\\
    \hline
    $25$ & $\{\pm1,\pm7\}$ &
    $30$ & $\{\pm1,\pm11\}$\\
    $32$ & $\{\pm1,\pm15\}$ &
    $33$ & $\{\pm1,\pm2,\pm4,\pm8,\pm16\}$\\
    $34$ & $\{\pm1,\pm9,\pm13,\pm15\}$ &
    $35$ & $\{\pm1,\pm6,\pm8,\pm13\}$\\
    $35$ & $\left<-1,4,6\right>$, $\#\Delta=12$ &
    $36$ & $\{\pm1,\pm17\}$\\
    $39$ & $\{\pm1,\pm5,\pm8,\pm14\}$ &
    $39$ & $\left<-1,4\right>$, $\#\Delta=12$\\
    $40$ & $\{\pm1,\pm3,\pm9,\pm13\}$ &
    $40$ & $\{\pm1,\pm7,\pm9,\pm17\}$\\
    $40$ & $\{\pm1,\pm9,\pm11,\pm19\}$ &
    $41$ & $\left<-1,2\right>$, $\#\Delta=20$\\
    $45$ & $\left<-1,4\right>$, $\#\Delta=12$ &
    $48$ & $\{\pm1,\pm5,\pm19,\pm23\}$\\
    $48$ & $\{\pm1,\pm7,\pm17,\pm23\}$ &
    $48$ & $\{\pm1,\pm11,\pm13,\pm23\}$\\
    $55$ & $\left<-1,4\right>$, $\#\Delta=20$ &
    $64$ & $\left<-1,9\right>$, $\#\Delta=16$\\
    $75$ & $\left<-1,4\right>$, $\#\Delta=20$ & &\\
     
    \hline
\end{tabular}
\end{center}
\end{thm}

\begin{thm}\label{CtetragonalQpentagonalthm}
    The intermediate modular curve $X_\Delta(N)$ is tetragonal over $\C$ and has $\Q$-gonality at least $5$ if and only if 
    $$(N,\Delta)\in\{(31,\{\pm1,\pm5,\pm6\}),(31,\{\pm1,\pm2,\pm4,\pm8,\pm15\})\}.$$
    Moreover, the $\Q$-gonality of both these curves is equal to $5$.
\end{thm}

\begin{thm}\label{Qpentagonalthm}
    The intermediate modular curve $X_\Delta(N)$ is pentagonal over $\Q$ and over $\C$ if and only if
    $$(N,\Delta)\in\{(44,\{\pm1,\pm5,\pm7,\pm9,\pm19\}),(125,\left<-1,4\right>)\}.$$
    For $N=125$, this group $\Delta$ has $50$ elements.
\end{thm}
This paper is organized as follows. 

In \Cref{sectionmain} we present the results needed to prove the above theorems. More precisely, in \Cref{Fpsection} we give lower bounds on the $\Q$-gonality of curves $X_\Delta(N)$ via $\F_p$-gonality, in \Cref{CSsection} we give lower bounds on the $\C$-gonality obtained using the Castelnuovo-Severi inequality (\Cref{tm:CS}), in \Cref{rationalmapsection} we give rational morphisms from $X_\Delta(N)$ to $\mathbb{P}^1$, and in \Cref{bettisection} we determine the $\C$-tetragonal curves $X_\Delta(N)$. After that, in \Cref{thmproofssection} we prove the main theorems of this paper.

For the reader's convenience, in \Cref{tablesection} at the end of the paper we put Tables \ref{tab:main1} and \ref{tab:main2}. In these tables we list the curves $X_\Delta(N)$ for all levels $N$ studied in the paper (the list of these levels $N$ is given and explained at the beginning of \Cref{sectionmain}). For these curves $X_\Delta(N)$ we also give their $\C$ and $\Q$-gonality with the links to all results used to determine the gonality of that curve.

A lot of the results in this paper rely on \texttt{Magma} \cite{magma} and Sage computations. It is well known that, for a curve $C$ of genus $g$ that is neither trigonal nor a smooth plane quintic, its canonical model consists of $\displaystyle\frac{(g-2)(g-3)}{2}$ quadrics. Since we are only interested in curves of genus $g\geq5$ (none of them are trigonal by \cite{Jeon2007}) and there are no intermediate modular curves that are smooth plane quintics by \cite[Theorem 1.1]{Anni2023}, the $\displaystyle\frac{(g-2)(g-3)}{2}$ quadrics will give canonical models of these curves $X_\Delta(N)$. We used a function vanishing\_quadratic\_forms() from Maarten Derickx's Sage package MD Sage to find these quadrics.

It should also be mentioned that David Zwyina's \texttt{Magma} function FindCanonicalModel() on
\begin{center}
    \url{https://github.com/davidzywina/ActionsOnCuspForms},
\end{center}
used in \cite{zywina20}, also gives canonical models of modular curves $X_\Gamma$ for groups $\Gamma\leq\GL_2(\Z/N\Z)$. However, this function is much slower than the MD Sage function vanishing\_quadratic\_forms() used here.

The codes that verify all computations in this paper can be found on
\begin{center}
    \url{https://github.com/orlic1/gonality_X_Delta}.
\end{center}
All computations were performed on the Euler server at the Department of Mathematics, University of Zagreb with a Intel Xeon W-2133 CPU running at 3.60GHz and with
64 GB of RAM.

\section{Acknowledgements}

Many thanks to Maarten Derickx for his assistance with the MD Sage package for SageMath which was essential to get the models of intermediate curves. Many of his other comments were useful to me as well. Additionally, code and data associated to the paper \cite{Rouse_Sutherland_Zureick-Brown_2022} by Jeremy Rouse, Andrew V. Sutherland, and David Zureick-Brown was used in \Cref{Fp2pointscomputation}. Their code can be found on
\begin{center}
    \url{https://github.com/AndrewVSutherland/ell-adic-galois-images/tree/209c2f888669785151174f472ea2c9eafb6daaa9}.
\end{center}
I am also grateful to Filip Najman for his helpful comments and suggestions.

\section{Results}\label{sectionmain}

In this section we list the known results that will be used to determine the $\C$ and $\Q$-gonality of intermediate modular curves. We first mention two obvious lower bounds for a curve $C$ defined over $\Q$:
\begin{align*}
    \textup{gon}_\C(C)&\leq\textup{gon}_\Q(C),\\
    \textup{gon}_{\F_p}(C)&\leq\textup{gon}_\Q(C).
\end{align*}
Here $p$ is a prime of good reduction for $C$. 

A very important result we will use throughout the paper is Poonen's \cite[Proposition A.1.]{Poonen2007}, stated below.

\begin{prop}[Poonen]\label{poonen}
    Let $X$ be a curve of genus $g$ over a field $k$.
    \begin{enumerate}[(i)]
        \item If $L$ is a field extension of $k$, then $\textup{gon}_L(X)\leq \textup{gon}_k(X)$.
        \item If $k$ is algebraically closed and $L$ is a field extension of $k$, then $\textup{gon}_L(X)=\textup{gon}_k(X)$.
        \item If $g\geq2$, then $\textup{gon}_k(X)\leq 2g-2$.
        \item If $g\geq2$ and $X(k)\neq\emptyset$, then $\textup{gon}_k(X)\leq g$.
        \item If $k$ is algebraically closed, then $\textup{gon}_k(X)\leq\frac{g+3}{2}$.
        \item If $\pi:X\to Y$ is a dominant $k$-rational map, then $\textup{gon}_k(X)\leq \deg \pi\cdot\textup{gon}_k(Y)$.
        \item If $\pi:X\to Y$ is a dominant $k$-rational map, then $\textup{gon}_k(X)\geq\textup{gon}_k(Y)$.
    \end{enumerate}
\end{prop}

Since all modular curves $X_\Delta(N)$ have at least one rational cusp, this result implies that their $\Q$-gonality is bounded from above by their genus. Moreover, if we have 
$$\{\pm1\}\subseteq \Delta_1\subseteq \Delta_2 \subseteq (\Z/N\Z)^\times,$$
then, due to the natural projections $$X_1(N)\to X_{\Delta_1}(N)\to X_{\Delta_2}(N)\to X_0(N)$$ and \Cref{poonen}(vii), we conclude that
\begin{align}\label{dominantmapinequality}
    \textup{gon}_\Q(X_1(N))\geq\textup{gon}_\Q(X_{\Delta_1}(N))&\geq\textup{gon}_\Q(X_{\Delta_2}(N))\geq\textup{gon}_\Q(X_0(N)),\\
    \textup{gon}_\C(X_1(N))\geq\textup{gon}_\C(X_{\Delta_1}(N))&\geq\textup{gon}_\C(X_{\Delta_2}(N))\geq\textup{gon}_\C(X_0(N)).
\end{align}

Therefore, when searching for tetragonal curves $X_\Delta(N)$, we may restrict ourselves to the levels $N$ for which the curve $X_0(N)$ has $\C$-gonality at most $4$. Similarly, when searching for $\Q$-pentagonal curves $X_\Delta(N)$, we may restrict ourselves to the levels $N$ for which the curve $X_0(N)$ has $\Q$-gonality at most $5$. These levels $N$ are listed in the following two theorems.

\begin{thm}[\cite{Ogg74, HasegawaShimura_trig, JeonPark05}]
    The curve $X_0(N)$ has $\C$-gonality at most $4$ if and only if
    \begin{align*}
        N\in\{&1-75,77-81,83,85,87-89,91,92,94-96,98,100,101,\\
        &103,104,107,109,111,119,121,125,131,142,143,167,191\}.
    \end{align*}
\end{thm}

\begin{thm}[\cite{NajmanOrlic22}]
    The curve $X_0(N)$ has $\Q$-gonality equal to $5$ if and only if $N=109$.
\end{thm}

We can also eliminate those levels $N$ for which the $\Q$-gonality of the curve $X_1(N)$ is at most $3$, namely \cite[Table 1]{derickxVH}
$$N\in\{1-16,18,20\}.$$

Therefore, there are only finitely many intermediate modular curves $X_\Delta(N)$ we need to deal with. Moreover, from the group structure of the group $(\Z/N\Z)^\times$, we can easily see that for
$$N\in\{22,23,46,47,59,83,94,107,167\},$$
there are actually no intermediate modular curves $X_\Delta(N)$ because in these cases $(\Z/N\Z)^\times\cong \Z/2p\Z$ for some prime $p$.

\subsection{$\F_p$-gonality}\label{Fpsection}

In this section we use the results on the $\F_p$-gonality to get a lower bound on the $\Q$-gonality of the modular curves $X_\Delta(N)$.

\begin{lem}\cite[Lemma 3.5]{NajmanOrlic22}\label{Fp2points}
    Let $C$ be a curve, $p$ a prime of good reduction for $C$, and $q$ a power of $p$. Suppose $\#C(\F_q)>d(q+1)$. Then $\textup{gon}_\Q(C)>d$.
\end{lem}

Let us take a prime $p\nmid N$. Then $p$ is a prime of good reduction for all curves $X_\Delta(N)$ and we can use \Cref{Fp2points} to obtain a lower bound on the $\Q$-gonality.

\begin{prop}\label{Fp2pointscomputation}
    The modular curve $X_\Delta(N)$ has $\Q$-gonality at least $6$ for the following values of $N$ and $\Delta$:

    \begin{center}
\begin{tabular}{|c|c|c|c|}
\hline
$N$ & $\Delta$ & $p$ & $\#X_\Delta(N)(\F_{p^2})$\\
    \hline

    $71$ & $\left<-1,5\right>$ & $5$ & $182$\\
    $78$ & $\left<-1,5,31\right>$ & $5$ & $192$\\
    $80$ & $\left<-1,3,49\right>$ & $3$ & $68$\\
    $88$ & $\left<-1,21,25\right>$ & $3$ & $68$\\
    $91$ & $\left<-1,2\right>$ & $2$ & $38$\\
    $96$ & $\left<-1,5\right>$ & $5$ & $160$\\
    $104$ & $\left<-1,3,25\right>$ & $3$ & $72$\\
    $104$ & $\left<-1,5,27\right>$ & $5$ & $192$\\
    $143$ & $\left<-1,8\right>$ & $5$ & $180$\\

    \hline
\end{tabular}
\end{center}
\end{prop}

\begin{proof}
    We use \texttt{Magma} to compute the number of $\F_{p^2}$ rational points on these curves $X_\Delta(N)$. It is easy to check that $\#X_\Delta(N)(\F_{p^2})>5(p^2+1)$ and we can use \Cref{Fp2points} with $q=p^2$ to finish the proof.
\end{proof}

Here we used Andrew Sutherland's \texttt{Magma} function GL2PointCount($\Gamma$, $q$) which, for $\Gamma\leq \GL_2(\Z/N\Z)$, returns the number of $\F_q$-rational points of $X_\Gamma$. For a given group $\Delta\leq(\Z/N\Z)^\times$ the corresponding group $\Gamma$ has the following generators:

\begin{align*}
&\begin{bmatrix}
a & 0\\
0 & 1
\end{bmatrix}
: a\textup{ is a generator of } \Delta,\\
&\begin{bmatrix}
1 & 0\\
0 & d
\end{bmatrix}
: d\textup{ is a generator of } (\Z/N\Z)^\times,\\
&\begin{bmatrix}
1 & 1\\
0 & 1
\end{bmatrix}.\\
\end{align*}

We can also directly obtain the lower bound on the $\F_p$-gonality of $X_\Delta(N)$ by checking that the dimensions of Riemann-Roch spaces of all degree $\leq d$ effective $\F_p$-rational divisors are equal to $1$. This is a finite task since the number of such divisors is finite. We can also use certain tricks (like the ones used in the following three propositions) to reduce the number of divisors that need to be checked.

\begin{prop}\label{Fp_gonality_deg4}
The $\F_p$-gonality of the curve $X_\Delta(N)$ at least $5$ for the following values of $N$ and $\Delta$:

\begin{center}
\begin{tabular}{|c|c|c|}
\hline
$N$ & $\Delta$ & $p$\\
    \hline
    
    $31$ & $\{\pm1,\pm5,\pm6\}$ & $7$\\
    $31$ & $\{\pm1,\pm2,\pm4,\pm8,\pm16\}$ & $2$\\
    $125$ & $\left<-1,4\right>$ & $2$\\

    \hline
\end{tabular}
\end{center}
\end{prop}

\begin{proof}
    Using \texttt{Magma}, we compute that there are no functions of degree $\leq4$ in $\F_p(X_\Delta(N))$. We are able to reduce the number of divisors that need to be checked by noting the following: If there exists a function $f$ over a field $k$ of a certain degree and if $c\in k$, then the function $g(x):=\frac{1}{f(x)-c}$ has the same degree and its polar divisor contains a $k$-rational point.

    In our case, this means that we only need to check the divisors of the form 
    \begin{itemize}
        \item $1+1+1+1$ - sums of $4$ $\F_p$-points,
        \item $1+1+2$ - sums of $2$ $\F_p$-points and an irreducible degree $2$ rational effective divisor,
        \item $1+3$ - sums of and $\F_p$ point and an irreducible degree $3$ rational effective divisor,
    \end{itemize}
    and we do not need to check the divisors of the form $4$ (irreducible degree $4$ rational effective divisors). 
    
    Notice that there is also no need to check the divisors of the form $1+1+1$, $1+2$, ... (i.e., divisors of degree $\leq3$). For example, since all divisors of the form $1+1+1+1$ have Riemann-Roch dimension equal to $1$, all divisors of the form $1+1+1$ will automatically have Riemann-Roch dimension equal to $1$.
\end{proof}

\begin{prop}\label{Fp_gonality_deg5}
The $\F_p$-gonality of the curve $X_\Delta(N)$ at least $6$ for the following values of $N$ and $\Delta$:

\begin{center}
\hspace*{-1.5cm}
\begin{longtable}{|c|c|c||c|c|c||c|c|c|}
\hline
\addtocounter{table}{-1}
$N$ & $\Delta$ & $p$ & $N$ & $\Delta$ & $p$ & $N$ & $\Delta$ & $p$\\
    \hline
    
    $29$ & $\{\pm1,\pm12\}$ & $3$ & $33$ & $\{\pm1,\pm10\}$ & $5$ & $34$ & $\{\pm1,\pm13\}$ & $3$\\
    $35$ & $\{\pm1,\pm11,\pm16\}$ & $3$ & $37$ & $\{\pm1,\pm10,\pm11\}$ & $3$ & $38$ & $\{\pm1,\pm7,\pm11\}$ & $3$\\
    $39$ & $\{\pm1,\pm16,\pm17\}$ & $5$ & $40$ & $\{\pm1,\pm19\}$ & $3$ & $41$ & $\left<-1,4\right>$ & $3$\\
    $41$ & $\{\pm1,\pm3,\pm9,\pm14\}$ & $2$ & $42$ & $\{\pm1,\pm5,\pm17\}$ & $5$ & $42$ & $\{\pm1,\pm13\}$ & $5$\\
    $43$ & $\{\pm1,\pm2\}$ & $3$ & $43$ & $\{\pm1,\pm6,\pm7\}$ & $3$ & $44$ & $\{\pm1,\pm21\}$ & $3$\\
    $45$ & $\{\pm1,\pm8,\pm17,\pm19\}$ & $2$ & $45$ & $\{\pm1,\pm14,\pm16\}$ & $2$ & $48$ & $\{\pm1,\pm23\}$ & $5$\\
    $49$ & $\{\pm1,\pm18,\pm19\}$ & $2$ & $51$ & $\left<-1,2\right>$ & $2$ & $52$ & $\left<-1,21\right>$ & $3$\\
    $52$ & $\left<-1,3\right>$ & $3$ & $53$ & $\left<-1,4\right>$ & $19$ & $54$ & $\{\pm1,\pm17,\pm19\}$ & $5$\\
    $55$ & $\left<-1,16\right>$ & $2$ & $55$ & $\{\pm1,\pm12,\pm21,\pm23\}$ & $2$ & $56$ & $\left<-1,3\right>$ & $3$\\
    $56$ & $\left<-1,9,15\right>$ & $3$ & $56$ & $\left<-1,5,9\right>$ & $3$ & $56$ & $\{\pm1,\pm13,\pm15,\pm27\}$ & $3$\\
    $57$ & $\left<-1,8,20\right>$ & $2$ & $57$ & $\left<-1,2\right>$ & $3$ & $58$ & $\left<-1,9\right>$ & $3$\\
    $60$ & $\{\pm1,\pm11,\pm19,\pm29\}$ & $7$ & $60$ & $\{\pm1,\pm11,\pm13,\pm23\}$ & $7$ & $60$ & $\{\pm1,\pm7,\pm11,\pm17\}$ & $7$\\
    $61$ & $\left<-1,4\right>$ & $5$ & $61$ & $\left<-1,8\right>$ & $2$ & $61$ & $\left<-1,29\right>$ & $2$\\
    $62$ & $\left<-1,27\right>$ & $2$ & $63$ & $\left<-1,4,5\right>$ & $2$ & $63$ & $\left<-1,8,20\right>$ & $2$\\
    $63$ & $\left<-1,8,10\right>$ & $2$ & $63$ & $\left<-1,5,8\right>$ & $2$ & $63$ & $\left<-1,2\right>$ & $2$\\
    $64$ & $\{\pm1,\pm15,\pm17,\pm31\}$ & $3$ & $65$ & $\left<-1,2,7\right>$ & $2$ & $65$ & $\left<-1,4,6\right>$ & $2$\\
    $65$ & $\left<-1,3,4\right>$ & $2$ & $65$ & $\left<-1,8,12\right>$ & $2$ & $66$ & $\left<-1,25\right>$ & $5$\\
    $67$ & $\left<-1,8\right>$ & $2$ & $68$ & $\left<-1,9\right>$ & $3$ & $69$ & $\left<-1.4\right>$ & $2$\\
    $70$ & $\left<-1,27\right>$ & $3$ & $70$ & $\left<-1,9\right>$ & $3$ 
    & $71$ & $\left<-1,20\right>$ & $3$\\
    $72$ & $\left<^-1,13,25\right>$ & $5$ & $72$ & $\left<-1,17,25\right>$ & $5$ & $72$ & $\left<-1,5\right>$ & $5$\\
    $73$ & $\left<-1,21\right>$ & $2$ & $73$ & $\left<-1,25\right>$ & $3$ & $74$ & $\left<-1,25\right>$ & $3$\\
    $75$ & $\left<-1,16\right>$ & $2$ & $77$ & $\left<-1,32\right>$ & $2$ & $77$ & $\left<-1,8\right>$ & $2$\\
    $77$ & $\left<-1,4\right>$ & $2$ & $78$ & $\left<-1,35,49\right>$ & $5$ & $79$ & $\left<-1,27\right>$ & $2$\\
    $80$ & $\left<-1,7,9\right>$ & $3$ & $80$ & $\left<-1,21,49\right>$ & $3$ & $81$ & $\left<-1,8\right>$ & $2$\\
    $85$ & $\left<-1,2,9\right>$ & $2$ & $85$ & $\left<-1,3,4\right>$ & $3$ & $87$ & $\left<-1,4\right>$ & $2$\\
    $88$ & $\left<-1,5\right>$ & $3$ & $88$ & $\left<-1,25,105\right>$ & $3$ & $89$ & $\left<-1,9\right>$ & $2$\\
    $91$ & $\left<-1,8,12\right>$ & $2$ & $91$ & $\left<-1,8,24\right>$ & $2$ & $91$ & $\left<-1,8,48\right>$ & $2$\\
    $91$ & $\left<-1,4,12\right>$ & $2$ & $92$ & $\left<-1,9\right>$ & $3$ & $95$ & $\left<-1,4\right>$ & $2$\\
    $95$ & $\left<-1,8\right>$ & $2$ & $96$ & $\left<-1,17,25\right>$ & $5$ & $96$ & $\left<-1,11,25\right>$ & $5$\\
    $100$ & $\left<-1,9\right>$ & $3$ & $101$ & $\left<-1,4\right>$ & $2$ & $103$ & $\left<-1,22\right>$ & $2$\\
    $104$ & $\left<-1,5,9\right>$ & $3$ & $109$ & $\left<-1,2\right>$ & $2$ & $109$ & $\left<-1,36\right>$ & $2$\\
    $111$ & $\left<-1,8\right>$ & $2$ & $111$ & $\left<-1,4\right>$ & $2$ & $119$ & $\left<-1,27\right>$ & $2$\\
    $119$ & $\left<-1,9\right>$ & $2$ & & & & & &\\

    \hline
\end{longtable}
\end{center}
\end{prop}

\begin{proof}
    Using \texttt{Magma}, we compute that there are no functions of degree $\leq5$ in $\F_p(X_\Delta(N))$. We are able to reduce the number of divisors that need to be checked by noting the following (a similar idea as in \cite[Propositions 5.13, 5.14]{NajmanOrlic22}): 
    
    Suppose that there exists a function $f:X_\Delta(N)\to\mathbb{P}^1$ over $\F_p$ of a certain degree and that $\#X_\Delta(N)(\F_p)>d(p+1)$. By the pigeonhole principle, it follows that there is a point $c\in \mathbb{P}^1(\F_3)$ such that $f^{-1}(c)$ contains at least $d+1$ $\F_p$-rational points. Therefore, the polar divisor of a function $\displaystyle g(x):=\frac{1}{f(x)-c}$ (which is also defined over $\F_p$ and has the same degree as $f$) contains at least $d+1$ $\F_p$-rational points.

    This means that we only need to check the divisors that contain at least $d+1$ $\F_p$-rational points. For example, if $d=2$, we only need to check the divisors of the form $1+1+1+1+1$ and $1+1+1+2$, and if $d=1$, we only need to check the divisors of the form $1+1+1+1+1$, $1+1+1+2$, and $1+1+3$.
\end{proof}

\begin{prop}\label{Fp_gonality_large}
    The $\F_p$-gonality of the curve $X_\Delta(N)$ is bounded from below by $d$ for the following values of $N$ and $\Delta$:

    \begin{center}
\begin{tabular}{|c|c|c|c|}
\hline
$N$ & $\Delta$ & $p$ & $d$\\
    \hline
    
    $35$ & $\{\pm1,\pm6\}$ & $2$ & $8$\\
    $37$ & $\{\pm1,\pm6\}$ & $2$ & $9$\\
    $39$ & $\{\pm1,\pm14\}$ & $2$ & $8$\\
    $40$ & $\{\pm1,\pm11\}$ & $3$ & $7$\\
    $40$ & $\{\pm1,\pm9\}$ & $3$ & $8$\\

    \hline
\end{tabular}
\end{center}
\end{prop}

\begin{proof}
    Using \texttt{Magma}, we compute that there are no functions of degree $\leq d-1$ in $\F_p(X_\Delta(N))$. We use the same idea as in \Cref{Fp_gonality_deg5}. 
\end{proof}

\begin{remark}
    Most computations for Propositions \ref{Fp_gonality_deg4}, \ref{Fp_gonality_deg5}, and \ref{Fp_gonality_large} were relatively fast (several minutes). However, there were some cases that took longer to finish, For example, the cases $(N,\Delta)=(78,\left<-1,35,49\right>),(96,\left<-1,11,25\right>),(104,\left<-1,5,9\right>)$ took around $2$ hours to finish.
    
    Also, the computation time for $(N,\Delta)=(53,\left<-1,4\right>)$ was around $4$ hours, mostly due to the larger field $\F_{19}$ the program was working with. This is because the number of Riemann-Roch spaces that need to be computed to give a lower bound of $d$ on the $\F_p$-gonality is $O(p^d)$ \cite[Section 8.1]{NajmanOrlic22} and it is therefore advisable to choose small values of $p$ when computing the $\F_p$-gonality.

    From this we can also see that the complexity grows exponentially with $d$, meaning that computing the $\F_p$-gonality using this method becomes more difficult and increasingly unfeasible as the gonality grows, especially in the high genus cases.  
\end{remark}

\subsection{Castelnuovo-Severi inequality}\label{CSsection}

This is a very useful tool for producing a lower bound on the gonality (see \cite[Theorem 3.11.3]{Stichtenoth09} for a proof).
\begin{prop}[Castelnuovo-Severi inequality] 
\label{tm:CS}
Let $k$ be a perfect field, and let $X,\ Y, \ Z$ be curves over $k$ with genera $g(X), g(Y), g(Z)$. Let non-constant morphisms $\pi_Y:X\rightarrow Y$ and $\pi_Z:X\rightarrow Z$ over $k$ be given, and let their degrees be $m$ and $n$, respectively. Assume that there is no morphism $X\rightarrow X'$ of degree $>1$ through which both $\pi_Y$ and $\pi_Z$ factor. Then the following inequality holds:
\begin{equation} \label{eq:CS}
g(X)\leq m \cdot g(Y)+n\cdot g(Z) +(m-1)(n-1).
\end{equation}
\end{prop}

Since $\C$ and $\Q$ are both perfect fields, we can use Castelnuovo-Severi inequality to get lower bounds on both $\C$ and $\Q$-gonalities.

As we can see, the assumption of the Castelnuovo-Severi inequality is that there is no morphism $X\to X'$ defined over $\overline{k}$ through which both $\pi_Y$ and $\pi_Z$ factor. However, Khawaja and Siksek have recently shown \cite[Theorem 14]{KhawajaSiksek2023} that we can weaken this assumption. Namely, that there is no morphism $X\to X'$ defined over $k$ through which both $\pi_Y$ and $\pi_Z$ factor.

\begin{prop}\label{CSprop}
The $\C$-gonality of the curve $X_\Delta(N)$ is at least $6$ for the following values of $N$ and $\Delta$:

\begin{center}
\begin{tabular}{|c|c|c|c|c|c|c|}
\hline
$N$ & $\Delta$ & \textup{LMFDB label of }$X_\Delta(N)$ & $g(X_\Delta(N))$ & $Y$ & $\deg$ & $g(Y)$\\
    \hline

$48$ & $\{\pm1,\pm7\}$ & $48.384.19.bj.1$ & $19$ & $X_{\{\pm1,\pm7,\pm17,\pm23\}}(48)$ & $2$ & $7$\\
$48$ & $\{\pm1,\pm17\}$ & $48.384.19.bc.1$ & $19$ & $X_{\{\pm1,\pm7,\pm17,\pm23\}}(48)$ & $2$ & $7$\\
$50$ & $\{\pm1,\pm7\}$ & $50.450.22.f.1$ & $22$ & $25.150.4.f.1$ & $3$ & $4$\\
$62$ & $\{\pm1,\pm5,\pm25\}$ & $62.480.31.c.1$ & $31$ & $31.160.6.c.1$ & $3$ & $6$\\
$72$ & $\{\pm1,\pm17,\pm19,\pm35\}$ & $72.432.21.tx.1$ & $21$ & $36.216.7.u.1$ & $2$ & $7$\\
$74$ & $\left<-1,23\right>$ & $74.342.22.b.1$ & $22$ & $37.114.4.b.2$ & $3$ & $4$\\
$98$ & $\left<-1,27\right>$ & $98.504.19.b.1$ & $19$ & $49.168.3.b.1$ & $3$ & $3$\\
    \hline
\end{tabular}
\end{center}
\end{prop}

\begin{proof}
    We know from \cite{IshiiMomose} and \cite{Jeon2007} that these curves $X_\Delta(N)$ are neither hyperelliptic nor trigonal. Therefore, their $\C$-gonality is at least $4$.

    Suppose that the curve $X_\Delta(N)$ is tetragonal for some values of $N$ and $\Delta$ from this table. This would mean that there is a degree $4$ morphism from $X_\Delta(N)$ to $\mathbb{P}^1$. We can easily check on LMFDB that there is a degree $\deg$ morphism from $X_\Delta(N)$ to the curve $Y$.
    
    If $\deg=3$, then these two morphisms surely do not factor through a morphism of degree $>1$. If $\deg=2$ and these two morphisms factor through a morphism $X_\Delta(N)\to X'$, then this morphism $X_\Delta(N)\to X'$ must be of degree $2$ and we must have $X'\cong Y$. Therefore, we must have 
    $$X_\Delta(N)\xrightarrow{2} Y\xrightarrow{2} \mathbb{P}^1,$$
    meaning that the curve $Y$ is hyperelliptic. However, we can easily check on LMFDB that this is not the case and we get a contradiction.

    This means that we can apply Castelnuovo-Severi inequality to these two morphisms to get
    $$g(X_\Delta(N))\leq 4\cdot0+\deg\cdot g(Y)+3(\deg-1).$$
    This inequality does not hold for these values of $N$ and $\Delta$, however, meaning that there is no degree $4$ morphism from $X_\Delta(N)$ to $\mathbb{P}^1$.

    Suppose now that the curve $X_\Delta(N)$ is pentagonal for some values of $N$ and $\Delta$ from the table. This would mean that there is a degree $5$ morphism from $X_\Delta(N)$ to $\mathbb{P}^1$. Since $\deg=2,3$ for all entries in the table, this hypothetical morphism and the degree $\deg$ morphism to $Y$ surely do not factor through a morphism of degree $>1$.

    This means that we can apply Castelnuovo-Severi inequality to these two morphisms to get
    $$g(X_\Delta(N))\leq 5\cdot0+\deg\cdot g(Y)+4(\deg-1).$$
    This inequality does not hold for these values of $N$ and $\Delta$, however, meaning that there is no degree $5$ morphism from $X_\Delta(N)$ to $\mathbb{P}^1$.
\end{proof}

\subsection{Rational morphisms to $\mathbb{P}^1$}\label{rationalmapsection}

In \Cref{Fpsection} and \Cref{CSsection}, we were giving lower bounds on the $\C$ and $\Q$-gonality of curves $X_\Delta(N)$. Now we give upper bounds by finding rational morphisms from $X_\Delta(N)$ to $\mathbb{P}^1$.

\begin{prop}\label{genus5gonalmap}
    There exists a degree $4$ rational morphism from $X_\Delta(N)$ to $\mathbb{P}^1$ for the following values of $N$ and $\Delta$:

\begin{center}
\begin{tabular}{|c|c|}
\hline
$N$ & $\Delta$\\
    \hline
    
$30$ & $\{\pm1,\pm11\}$\\
$32$ & $\{\pm1,\pm15\}$\\
$33$ & $\{\pm1,\pm2,\pm4,\pm8,\pm16\}$\\
$35$ & $\{\pm1,\pm4,\pm6,\pm9,\pm11,\pm16\}$\\
$36$ & $\{\pm1,\pm17\}$\\
$39$ & $\{\pm1,\pm4,\pm10,\pm14,\pm16,\pm17\}$\\
$40$ & $\{\pm1,\pm9,\pm11,\pm19\}$\\
$41$ & $\left<-1,2\right>$\\
$45$ & $\{\pm1,\pm4,\pm11,\pm14,\pm16,\pm19\}$\\
$64$ & $\left<-1,9\right>$\\

    \hline
\end{tabular}
\end{center}
    
\end{prop}

\begin{proof}
    All these curves $X_\Delta(N)$ are of genus $5$ and from the discussion at the end of the Introduction we know that their canonical models are intersections of three quadrics. Using Sage, we obtained canonical models for these curves. With these models, we used a \texttt{Magma} function Genus5GonalMap(C) which returned that the $\C$-gonality of the curve $X_\Delta(N)$ is $4$ and also gave the equations of this degree $4$ morphism. The equations were all defined over $\Q$, therefore these curves $X_\Delta(N)$ are all $\Q$-tetragonal.
\end{proof}

\begin{remark}
    \texttt{Magma} has inbuilt functions Genus2GonalMap(C), Genus3GonalMap(C), Genus4GonalMap(C), Genus5GonalMap(C), and Genus6GonalMap(C) which return the $\C$-gonality of the curve $C$ and a gonal map to $\mathbb{P}^1$, defined over some number field.
    
    However, these functions seem to expect the 'usual' model of $C$. For example, they expect the model of a genus $3$ curve to be a single quartic and a model of a genus $5$ curves to be an intersection of three quadrics. Otherwise, Magma can return an error in the current version of \texttt{Magma} (V2.28-9.). 
    
    For example, for genus $5$ quotients of the modular curve $X_0(N)$, the inbuilt \texttt{Magma} function X0NQuotient(C) returns a model that is an intersection of cubics instead of an intersection of three quadrics.
\end{remark}

\begin{prop}\label{magmamapprop}
    There exists a degree $d$ rational morphism from $X_\Delta(N)$ to $\mathbb{P}^1$ for the following values of $N$ and $\Delta$:

\begin{center}
\begin{tabular}{|c|c|c|}
\hline
$N$ & $\Delta$ & $d$ \\
    \hline
    
$29$ & $\{\pm1,\pm12\}$ & $6$\\
$31$ & $\{\pm1,\pm2\}$ & $5$\\
$31$ & $\{\pm1,\pm5,\pm6\}$ & $5$\\
$33$ & $\{\pm1,\pm10\}$ & $6$\\
$34$ & $\{\pm1,\pm13\}$ & $6$\\
$35$ & $\{\pm1,\pm11,\pm16\}$ & $6$\\
$35$ & $\{\pm1,\pm6,\pm8,\pm13\}$ & $4$\\
$37$ & $\{\pm1,\pm10,\pm11\}$ & $6$\\
$39$ & $\{\pm1,\pm16,\pm17\}$ & $6$\\
$40$ & $\{\pm1,\pm11\}$ & $7$\\
$44$ & $\{\pm1,\pm5,\pm7,\pm9,\pm19\}$ & $5$\\
$55$ & $\left<-1,4\right>$ & $4$\\

    \hline
\end{tabular}
\end{center}
    
\end{prop}

\begin{proof}
    We used Sage to find a canonical model for these curves $X_\Delta(N)$. After that, we used \texttt{Magma} to find a degree $d$ rational effective divisor with Riemann-Roch dimension at least $2$. In all cases except for $(N,\Delta)=(37,\{\pm1,\pm10,\pm11\})$ this divisor was a sum of $d$ rational points. 
    
    For $(N,\Delta)=(37,\{\pm1,\pm10,\pm11\})$ we were not able to find a degree $6$ function whose polar divisor is supported on rational points so we had to search for quadratic points.

    We searched for quadratic points by intersecting the curve with hyperplanes of the form
    $$b_0x_0+\ldots+b_kx_k=0,$$
    where $b_0,\ldots,b_k\in \Z$ are coprime and chosen up to a certain bound, a similar idea as in \cite[Section 3.2]{Box19}. Note that, in a quadratic point $(x_0,\ldots,x_k)$, already its first three coordinates must be linearly dependent over $\Q$. Therefore, it is enough to check the hyperplanes
    $$b_0x_0+b_1x_1+b_2x_2=0.$$
    We used these quadratic points to find a degree $6$ effective rational divisor with Riemann-Roch dimension $2$. This divisor was a sum of $4$ rational points and an irreducible degree $2$ rational divisor
\end{proof}

\begin{prop}\label{quotientmapprop}
    There exists a degree $d$ rational morphism from $X_\Delta(N)$ to $\mathbb{P}^1$ for the following values of $N$ and $\Delta$:

\begin{center}
\begin{tabular}{|c|c|c|c|c|c|c|}
\hline
$N$ & $\Delta$ & $d$ & \textup{LMFDB label of }$X_\Delta(N)$ & $Y$ & $\deg$ & $\textup{gon}_\Q(Y)$\\
    \hline
    
$33$ & $\{\pm1,\pm2,\pm4,\pm8,\pm16\}$ & $4$ & $33.96.5.a.4$ & $X_0(33)$ & $2$ & $2$\\
$34$ & $\{\pm1,\pm13\}$ & $6$ & $34.216.9.a.1$ & $17.72.1.a.2$ & $3$ & $2$\\
$35$ & $\{\pm1,\pm6\}$ & $8$ & $35.288.13.a.2$ & $X_{\{\pm1,\pm6,\pm8,\pm13\}}(35)$ & $2$ & $4$\\
$36$ & $\{\pm1,\pm17\}$ & $4$ & $36.216.7.u.1$ & $X_{\pm1}(18)$ & $2$ & $2$\\
$37$ & $\{\pm1,\pm6\}$ & $9$ & $37.342.16.c.2$ & $X_{\{\pm1,\pm6,\pm8,\pm10,\pm11,\pm14\}}(37)$ & $3$ & $3$\\
$38$ & $\{\pm1,\pm7,\pm11\}$ & $6$ & $38.180.10.a.1$ & $19.60.1.a.2$ & $3$ & $2$\\
$39$ & $\{\pm1,\pm5,\pm8,\pm14\}$ & $4$ & $39.168.9.a.1$ & $13.42.0.a.2$ & $4$ & $1$\\
$39$ & $\{\pm1,\pm14\}$ & $8$ & $39.336.17.c.1$ & $X_{\pm1}(13)$ & $4$ & $2$\\
$40$ & $\{\pm1,\pm3,\pm9,\pm13\}$ & $4$ & $40.144.7.fp.1$ & $X_0(40)$ & $2$ & $2$\\
$40$ & $\{\pm1,\pm7,\pm9,\pm17\}$ & $4$ & $40.144.7.fs.1$ & $X_0(40)$ & $2$ & $2$\\
$40$ & $\{\pm1,\pm19\}$ & $6$ & $40.288.9.bh.1$ & $X_{\pm1}(20)$ & $2$ & $3$\\
$40$ & $\{\pm1,\pm9\}$ & $8$ & $40.288.13.sp.1$ & $X_{\{\pm1,\pm9,\pm11,\pm19\}}(40)$ & $2$ & $4$\\
$48$ & $\{\pm1,\pm5,\pm19,\pm23\}$ & $4$ & $48.192.7.hj.1$ & $X_0(48)$ & $2$ & $2$\\
$48$ & $\{\pm1,\pm7,\pm17,\pm23\}$ & $4$ & $48.192.7.ho.1$ & $X_0(48)$ & $2$ & $2$\\
$125$ & $\left<-1,4\right>$ & $5$ & $125.300.16.a.1$ & $25.60.0.a.1$ & $5$ & $1$\\

    \hline
\end{tabular}
\end{center}
    
\end{prop}

\begin{proof}
    This degree $d$ morphism is obtained as a composition map
    $$X_\Delta(N)\xrightarrow{\deg}Y\xrightarrow{\textup{gon}_\Q(Y)}\mathbb{P}^1.$$
    The map from $X_\Delta(N)$ to $Y$ a rational projection map and can be checked on LMFDB. It only remains to discuss the $\Q$-gonality of the curve $Y$. 
    
    If $Y$ is of genus $0,1$ or is hyperelliptic, this is obvious. The $\Q$-gonalities of the curves $X_{\{\pm1,\pm6,\pm8,\pm13\}}(35)$ and $X_{\{\pm1,\pm9,\pm11,\pm19\}}(40)$ were proved in Propositions \ref{magmamapprop} and \ref{genus5gonalmap}. The $\Q$-gonality of the curve $X_{\{\pm1,\pm6,\pm8,\pm10,\pm11,\pm14\}}(37)$ was proved in \Cref{CtrigonalQtetragonalthm} and the $\Q$-gonalities of the curves $X_{\pm1}(N)\cong X_1(N)$ were proved in \cite{derickxVH}.
\end{proof}

\subsection{$\C$-gonalities}\label{bettisection}

In this section we will determine the cases when $\textup{gon}_\C(X_\Delta(N))=4$. First, we present the lower bound on the $\C$-gonality of any modular curve by Kim and Sarnak \cite[Appendix 2]{Kim2002}, mentioned in the Introduction.

\begin{thm}\cite[Theorem 1.2.]{JeonKimPark06}\label{kimsarnakbound}
    Let $X_\Gamma$ be the algebraic curve corresponding to a congruence subgroup $\Gamma\subseteq \SL_2(\Z)$ of index
    $$D_\Gamma=[\SL_2(\Z):\pm\Gamma].$$
    If $X_\Gamma$ is $d$-gonal, then $D_\Gamma\leq \frac{12000}{119}d$.
\end{thm}

\begin{remark}
    From Kim and Sarnak's arguments of \cite{Kim2002} we can get that the constant $\frac{12000}{119}$ can be replaced with a slightly better constant $\frac{2^{15}}{325}$. However, their difference is $\sim0.01572$ which does not make any difference in our calculations (we would need to work with gonality $>100$ for it to make a difference).
\end{remark}

\begin{cor}\label{indexprop}
    The modular curve $X_\Delta(N)$ has $\C$-gonality at least $6$ for the following values of $N$ and $\Delta$:
    \begin{center}
\begin{tabular}{|c|c|c||c|c|c|}
\hline
$N$ & $\Delta$ & $[\SL_2(\Z):\Delta]$ & $N$ & $\Delta$ & $[\SL_2(\Z):\Delta]$ \\
    \hline

$53$ & $\left<-1,2^{13}\right>$ & $702$ &
$58$ & $\left<-1,3^7\right>$ & $630$\\
$66$ & $\left<-1,5^5\right>$ & $720$ &
$67$ & $\left<-1,2^{11}\right>$ & $748$\\
$69$ & $\left<-1,2^{11}\right>$ & $1056$ &
$75$ & $\left<-1,2^5\right>$ & $600$\\
$79$ & $\left<-1,3^{13}\right>$ & $1040$ &
$87$ & $\left<-1,2^7\right>$ & $840$\\
$88$ & $\left<-1,21,5^5\right>$ & $720$ &
$89$ & $\left<-1,3^{11}\right>$ & $990$\\
$92$ & $\left<-1,3^{11}\right>$ & $1584$ &
$98$ & $\left<-1,3^7\right>$ & $1176$\\
$100$ & $\left<-1,3^5\right>$ & $900$ &
$101$ & $\left<-1,2^5\right>$ & $510$\\
$103$ & $\left<-1,5^{17}\right>$ & $1763$ &
$121$ & $\left<-1,32\right>$ & $660$\\
$121$ & $\left<-1,2^{11}\right>$ & $1452$ &
$125$ & $\left<-1,32\right>$ & $750$\\
$131$ & $\left<-1,2^5\right>$ & $660$ &
$131$ & $\left<-1,2^{13}\right>$ & $1716$\\
$142$ & $\left<-1,7^5\right>$ & $1080$ &
$142$ & $\left<-1,7^7\right>$ & $1512$\\
$143$ & $\left<-1,2^5\right>$ & $840$ &
$191$ & $\left<-1,19^5\right>$ & $960$\\
$191$ & $\left<-1,19^{19}\right>$ & $3648$ & & & \\

    \hline
\end{tabular}
\end{center}
\end{cor}

\begin{proof}
    From \Cref{kimsarnakbound}, we can see that $X_\Delta(N)$ cannot be $d$-gonal for $d\leq5$ because $[\SL_2(\Z):\Delta]>\lfloor5\cdot\frac{12000}{119}\rfloor=504$ in all these cases.
\end{proof}

\begin{prop}\label{tetragonalgenus6}
    The curves $X_\Delta(N)$ are $\C$-tetragonal for
    $$(N,\Delta)\in\{(31,\{\pm1,\pm5,\pm6\}),(31,\{\pm1,\pm2,\pm4,\pm8,\pm15\})\}.$$
\end{prop}

\begin{proof}
    These curves are of genus $6$. Therefore, they have $\C$-gonality at most $4$ by \Cref{poonen}(v). We also know that they are neither hyperelliptic nor trigonal by \cite{IshiiMomose} and \cite{Jeon2007}. Hence their $\C$-gonality is equal to $4$.
\end{proof}

We now state the Tower Theorem \cite[Theorem 2.1]{NguyenSaito}.

\begin{thm}[The Tower Theorem]\label{towerthm}
Let $C$ be a curve defined over a perfect field $k$ such that $C(k)\neq0$ and let $f:C\to \mathbb{P}^1$ be a non-constant morphism over $\overline{k}$ of degree $d$. Then there exists a curve $C'$ defined over $k$ and a non-constant morphism $C\to C'$ defined over $k$ of degree $d'$ dividing $d$ such that the genus of $C'$ is $\leq (\frac{d}{d'}-1)^2$.
\end{thm}

\begin{cor}\cite[Corollary 4.6. (ii)]{NajmanOrlic22}\label{towerthmcor}
    Let $C$ be a curve defined over $\Q$ with $\textup{gon}_\C (C)=4$ and $g(C)\geq10$ and such that $C(\Q)\neq\emptyset$. Then $\textup{gon}_\Q (C)=4$.
\end{cor}

From \Cref{towerthm} and \Cref{towerthmcor} we see that in order to prove that a curve of genus $g\geq10$ has $\C$-gonality at least $5$, it is enough to prove that its $\Q$-gonality is at least $5$. All such non-tetragonal curves $X_\Delta(N)$ have been dealt with in Sections \ref{Fpsection}, \ref{CSsection}, and \Cref{indexprop}. Therefore, the remaining curves are those with genus $\leq9$.

In order to prove that these curves are not tetragonal, we will use graded Betti numbers $\beta_{i,j}$. We will follow the notation in \cite[Section 1.]{JeonPark05}. The results we mention can be found there and in \cite[Section 3.1.]{NajmanOrlic22}.

\begin{definition}
    For a curve $X$ and divisor $D$ of degree $d$, $g_d^r$ is a subspace $V$ of the Riemann-Roch space $L(D)$ such that $\dim V=r+1$.
\end{definition}

Therefore, we want to determine whether $X_0^+(N)$ has a $g_4^1$. Green's conjecture relates graded Betti numbers $\beta_{i,j}$ with the existence of $g_d^r$.

\begin{conj}[Green, \cite{Green84}]
    Let $X$ be a curve of genus $g$. Then $\beta_{p,2}\neq 0$ if and only if there exists a divisor $D$ on $X$ of degree $d$ such that a subspace $g_d^r$ of $L(D)$ satisfies $d\leq g-1$, $r=\ell(D)-1\geq1$, and $d-2r\leq p$.
\end{conj}

The "if" part of this conjecture has been proven in the same paper.

\begin{thm}[Green and Lazarsfeld, Appendix to \cite{Green84}]\label{thmGreenLazarsfeld}
    Let $X$ be a curve of genus $g$. If $\beta_{p,2}=0$, then there does not exist a divisor $D$ on $X$ of degree $d$ such that a subspace $g_d^r$ of $L(D)$ satisfies $d\leq g-1$, $r\geq1$, and $d-2r\leq p$.
\end{thm}

\begin{cor}\cite[Corollary 3.20]{Orlic2023}\label{betti22cor}
    Let $X$ be a curve of genus $g\geq5$ with $\beta_{2,2}=0$. Suppose that $X$ is neither hyperelliptic nor trigonal. Then $\textup{gon}_\C(X)\geq5$.
\end{cor}

\begin{prop}\label{bettiprop}
    The modular curve $X_\Delta(N)$ has $\C$-gonality at least $5$ for the following values of $N$ and $\Delta$:
    \begin{center}
    \begin{tabular}{|c|c||c|c||c|c|}
\hline
$N$ & $\Delta$ & $N$ & $\Delta$ & $N$ & $\Delta$ \\
    \hline

$29$ & $\{\pm1,\pm12\}$ & $34$ & $\{\pm1,\pm13\}$ & $35$ & $\{\pm1,\pm11,\pm16\}$\\
$39$ & $\{\pm1,\pm16,\pm17\}$ & $40$ & $\{\pm1,\pm19\}$ & $42$ & $\{\pm1,\pm5,\pm17\}$\\
$43$ & $\left<-1,2\right>$ & $44$ & $\{\pm1,\pm5,\pm7,\pm9,\pm19\}$ & $45$ & $\{\pm1,\pm14,\pm16\}$\\
$51$ & $\left<-1,2\right>$ & $52$ & $\left<-1,3\right>$ & $53$ & $\left<-1,4\right>$\\
$55$ & $\left<-1,4\right>$ & $56$ & $\left<-1,3\right>$ & $57$ & $\left<-1,2\right>$\\
$61$ & $\left<-1,4\right>$ & $63$ & $\left<-1,4,5\right>$ & $65$ & $\left<-1,2,7\right>$\\
$72$ & $\left<-1,11,25\right>$ & $73$ & $\left<-1,25\right>$ & &\\

    \hline
\end{tabular}
\end{center}
\end{prop}

\begin{proof}
    In all these cases we use \texttt{Magma} to compute $\beta_{2,2}=0$ and the result follows from \Cref{betti22cor} since these curves are neither hyperelliptic nor trigonal.
\end{proof}

\section{Proofs of the main theorems}\label{thmproofssection}

By the results of Ishii and Momose \cite{IshiiMomose} and Jeon and Kim \cite{Jeon2007}, we know that the only hyperelliptic curve $X_\Delta(N)\neq X_0(N),X_1(N)$ is a curve $X_{\{\pm1,\pm8\}}(21)$ and that all $\C$-trigonal curves $X_\Delta(N)\neq X_0(N),X_1(N)$ are of genus $3$ or $4$. Moreover, \Cref{CtrigonalQtetragonalthm} tells us which genus $4$ curves are $\Q$-trigonal (we know from \Cref{poonen}(iv) that genus $3$ curves have $\Q$-gonality $\leq3$). In the proofs of the main theorems, we may now suppose that all remaining curves we consider have $\C$-gonality at least $4$.

\begin{proof}[Proof of \Cref{Qtetragonalthm}]
    The curve $X_{\{\pm1,\pm7\}}$ is tetragonal over $\Q$ due to \Cref{CtrigonalQtetragonalthm}. Other listed curves are tetragonal over $\Q$ due to Propositions \ref{genus5gonalmap}, \ref{magmamapprop}, and \ref{quotientmapprop}.

    Suppose now that $X_\Delta(N)$ is a curve of $\C$-gonality at least $4$ not listed in the theorem. In the first case, we prove that the $\Q$-gonality of this curve is at least $5$ in Propositions \ref{Fp2pointscomputation}, \ref{Fp_gonality_deg4}, \ref{Fp_gonality_deg5}, \ref{Fp_gonality_large}, and \Cref{CSprop}, Otherwise, there is a projection map
    $$X_\Delta(N)\to X_{\Delta_1}(N)$$
    to a curve $X_{\Delta_1}(N)$ for which we have already proven that it is not $\Q$-tetragonal. In that case we can use \Cref{poonen}(vii) to prove that the curve $X_\Delta(N)$ is not $\Q$-tetragonal.

    For example, we consider the case $N=41$. Since $(\Z/41\Z)^\times\cong \Z/40\Z$, there are $4$ intermediate modular curves $X_\Delta(41)$, namely
    $$\Delta\in\{\{\pm1,\pm9\},\{\pm1,\pm3,\pm9,\pm14\},\{\pm1,\pm4,\pm10,\pm16,\pm18\},\left<-1,2\right>\}.$$
    The curve with $\Delta=\left<-1,2\right>$ is $\Q$-tetragonal due to \Cref{genus5gonalmap} and the curves with $\Delta=\{\pm1,\pm3,\pm9,\pm14\},\{\pm1,\pm4,\pm10,\pm16,\pm18\}$ are not $\Q$-tetragonal due to \Cref{Fp_gonality_deg5}. Then the curve with $\Delta=\{\pm1,\pm9\}$ is automatically not $\Q$-tetragonal since it maps to a curve with $\Delta=\{\pm1,\pm3,\pm9,\pm14\}$.
\end{proof}

\begin{proof}[Proof of \Cref{CtetragonalQpentagonalthm}]
    The two listed curves are $\C$-tetragonal due to \Cref{tetragonalgenus6} and they admit a degree $5$ rational morphism to $\mathbb{P}^1$ due to \Cref{magmamapprop}. However, they are not $\Q$-tetragonal due to \Cref{Fp_gonality_deg4}.

    Suppose now that $X_\Delta(N)$ is a curve of $\C$-gonality at least $4$ not listed in Theorems \Cref{Qtetragonalthm} and \Cref{CtetragonalQpentagonalthm}. Then we either prove that its $\C$-gonality is at least $5$ using \Cref{CSprop}, \Cref{indexprop}, \Cref{bettiprop}, the bound on $\F_p$-gonality in \Cref{Fpsection} along with \Cref{towerthmcor} for curves of genus $g\geq10$, or we use \Cref{poonen}(vii) with a map $X_\Delta(N)\to X_{\Delta_1}(N)$, similarly as in the proof of the previous theorem.
\end{proof}

\begin{proof}[Proof of \Cref{Qpentagonalthm}]
    The two listed curves are $\Q$-pentagonal due to Propositions \ref{magmamapprop} and \ref{quotientmapprop}. We proved that the curve with $N=44$ has $\C$-gonality at least $5$ in \Cref{bettiprop} and that the curve with $N=125$ has $\Q$-gonality at least $5$ in \Cref{Fp_gonality_deg4}. Since the genus of the curve $X_{\left<-1,4\right>}(125)$ is equal to $16$, we can use \Cref{towerthmcor} to prove that its $\C$-gonality is at least $5$.

    Suppose now that $X_\Delta(N)$ is a curve of $\C$-gonality at least $4$ not listed in Theorems \ref{Qtetragonalthm}, \ref{CtetragonalQpentagonalthm}, and \ref{Qpentagonalthm}. Then we know that its $\C$-gonality is at least $5$. It remains to prove that its $\Q$-gonality is at least $6$.
    
    In the first case, we either use Propositions \ref{Fp_gonality_deg5}, \ref{Fp_gonality_large}, \ref{CSprop}, or \Cref{indexprop} to prove that the $\Q$-gonality is at least $6$. Otherwise, we use \Cref{poonen}(vii) with a map $X_\Delta(N)\to X_{\Delta_1}(N)$, similarly as in the previous two proofs.
\end{proof}

\section{Summary of results}\label{tablesection}

We summarize our results in the following tables. For each level $N$, there are $7$ entries, listed in the order they appear in: the structure of the group $(\Z/N\Z)^\times$ along with its generators, the group $\Delta$, the genus $g$ of $X_\Delta(N)$, the $\Q$-gonality of $X_\Delta(N)$, the $\C$-gonality of $X_\Delta(N)$ (if determined), and all results used to obtain the gonalities.

The curves for $N\leq40$ are listed in \Cref{tab:main1} since \Cref{Qgonalitythm} gives all $\Q$-gonalities of these curves. We only list the curves $X_\Delta(N)$ with genus $g\geq3$.

The curves for $N\geq41$ are listed in \Cref{tab:main2}. Some information about the genera of curves $X_\Delta(N)$ was taken from \cite[Table 6]{JEON2020272}. Some groups $\Delta$ have been omitted from the table, especially for bigger $N$ when there are many groups $\Delta$, since we can use \Cref{poonen}(vii) for them to prove that they are neither $\C$-tetragonal nor $\Q$-pentagonal. For non-tetragonal curves with genus $g\geq10$ we also omit their $\C$-gonality if the bound $\textup{gon}_\C\geq5$ can be deduced from \Cref{towerthmcor}.

\clearpage
\begin{center}
\begin{longtable}{|c|c|c|c|c|c|c|c|}
\caption{The $\Q$-gonalities of curves $X_\Delta(N)$ for $N\leq40$.}
\label{tab:main1}\\
  \hline

  $N$ & $(\Z/N\Z)^\times$ & generators & $\Delta$ & $g$ & $\textup{gon}_\Q$ & $\textup{gon}_\C$ & results used\\
  \hline

    $21$ & $C_2\times C_6$ & $-1,2$ & $\{\pm1,\pm8\}$ & $3$ & $2$ & $2$ & \cite{IshiiMomose}, \cite{Jeon2007}\\

    \hline
    
    $24$ & $C_2\times C_2\times C_2$ & $5,7,13$ & $\{\pm1,\pm5\}$ & $3$ & $3$ & $3$ & $g=3$, \ref{poonen}(iv)\\
     & & & $\{\pm1,\pm7\}$ & $3$ & $3$ & $3$ & $g=3$, \ref{poonen}(iv)\\

     \hline

     $25$ & $C_{20}$ & $2$ & $\{\pm1,\pm7\}$ & $4$ & $4$ & $3$ & \ref{CtrigonalQtetragonalthm}\\

     \hline

    $26$ & $C_{12}$ & $7$ & $\{\pm1,\pm5\}$ & $4$ & $3$ & $3$ & \ref{CtrigonalQtetragonalthm}\\
     & & & $\{\pm1,\pm3,\pm9\}$ & $4$ & $3$ & $3$ & \ref{CtrigonalQtetragonalthm}\\

     \hline

     $28$ & $C_2\times C_6$ & $3,13$ & $\{\pm1,\pm13\}$ & $4$ & $3$ & $3$ & \ref{CtrigonalQtetragonalthm}\\
      & & & $\{\pm1,\pm3,\pm9\}$ & $4$ & $3$ & $3$ & \ref{CtrigonalQtetragonalthm}\\

      \hline

      $29$ & $C_{28}$ & $2$ & $\{\pm1,\pm12\}$ & $8$ & $6$ & $5$ & \ref{Fp_gonality_deg5}, \ref{magmamapprop}, \ref{bettiprop}\\
       & & & $\left<-1,4\right>$ & $4$ & $3$ & $3$ & \ref{CtrigonalQtetragonalthm}\\

       \hline

       $30$ & $C_2\times C_4$ & $7,11$ & $\{\pm1,\pm11\}$ & $5$ & $4$ & $4$ & \ref{genus5gonalmap}\\

       \hline

       $31$ & $C_{30}$ & $3$ & $\{\pm1,\pm5,\pm6\}$ & $6$ & $5$ & $4$ & \ref{Fp_gonality_deg4}, \ref{magmamapprop}, \ref{tetragonalgenus6}\\
        & & & $\{\pm1,\pm2,\pm4,\pm8,\pm15\}$ & $6$ & $5$ & $4$ & \ref{Fp_gonality_deg4}, \ref{magmamapprop}, \ref{tetragonalgenus6}\\

        \hline

        $32$ & $C_2\times C_8$ & $-1,3$ & $\{\pm1,\pm15\}$ & $5$ & $4$ & $4$ & \ref{genus5gonalmap}\\

        \hline

        $33$ & $C_2\times C_{10}$ & $2,10$ & $\{\pm1,\pm10\}$ & $11$ & $6$ & $\geq5$ & \ref{Fp_gonality_deg5}, \ref{magmamapprop}, \ref{towerthmcor}\\
         & & & $\{\pm1,\pm2,\pm4,\pm8,\pm16\}$ & $5$ & $4$ & $4$ & \ref{genus5gonalmap}\\

         \hline

         $34$ & $C_{16}$ & $3$ & $\{\pm1,\pm13\}$ & $9$ & $6$ & $\geq5$ & \ref{Fp_gonality_deg5}, \ref{magmamapprop}, \ref{bettiprop}\\
          & & & $\{\pm1,\pm9,\pm13,\pm15\}$ & $5$ & $4$ & $4$ & \ref{genus5gonalmap}\\

          \hline

          $35$ & $C_2\times C_{12}$ & $2,6$ & $\{\pm1,\pm6\}$ & $13$ & $8$ & $\geq5$ & \ref{Fp_gonality_large}, \ref{quotientmapprop}, \ref{towerthmcor}\\
           & & & $\{\pm1,\pm11,\pm16\}$ & $9$ & $6$ & $\geq5$ & \ref{Fp_gonality_deg5}, \ref{magmamapprop} \ref{betti22cor}\\
           & & & $\{\pm1,\pm6,\pm8,\pm13\}$ & $7$ & $4$ & $4$ & \ref{magmamapprop}\\
           & & & $\left<-1,4,6\right>$ & $5$ & $4$ & $4$ & \ref{genus5gonalmap}\\

           \hline

           $36$ & $C_2\times C_6$ & $5,19$ & $\{\pm1,\pm17\}$ & $7$ & $4$ & $4$ & \ref{quotientmapprop}\\
            & & & $\{\pm1,\pm11,\pm13\}$ & $3$ & $3$ & $3$ & $g=3$, \ref{poonen}(iv)\\

            \hline

            $37$ & $C_{36}$ & $2$ & $\{\pm1,\pm6\}$ & $16$ & $9$ & $\geq5$ & \ref{Fp_gonality_large}, \ref{quotientmapprop}\\
            & & & $\{\pm1,\pm10,\pm11\}$ & $10$ & $6$ & $\geq5$ & \ref{Fp_gonality_deg5}, \ref{magmamapprop}, \ref{towerthmcor}\\
            & & & $\left<-1,8\right>$ & $4$ & $3$ & $3$ & \ref{CtrigonalQtetragonalthm}\\
            & & & $\left<-1,4\right>$ & $4$ & $3$ & $3$ & \ref{CtrigonalQtetragonalthm}\\

            \hline

            $38$ & $C_{18}$ & $3$ & $\{\pm1,\pm7,\pm11\}$ & $10$ & $6$ & $\geq5$ & \ref{Fp_gonality_deg5}, \ref{quotientmapprop}, \ref{towerthmcor}\\

            \hline

            $39$ & $C_2\times C_{12}$ & $-1,2$ & $\{\pm1,\pm14\}$ & $17$ & $8$ & $\geq5$ & \ref{Fp_gonality_large}, \ref{quotientmapprop}, \ref{towerthmcor}\\
            & & & $\{\pm1,\pm16,\pm17\}$ & $9$ & $6$ & $\geq5$ & \ref{Fp_gonality_deg5}, \ref{magmamapprop}, \ref{bettiprop}\\
            & & & $\{\pm1,\pm5,\pm8,\pm14\}$ & $9$ & $4$ & $4$ & \ref{quotientmapprop}\\
            & & & $\left<-1,4\right>$ & $5$ & $4$ & $4$ & \ref{genus5gonalmap}\\

            \hline

            $40$ & $C_2\times C_2\times C_4$ & $-1,3,11$ & $\{\pm1,\pm9\}$ & $13$ & $8$ & $\geq5$ & \ref{Fp_gonality_large}, \ref{quotientmapprop}, \ref{towerthmcor}\\
            & & & $\{\pm1,\pm11\}$ & $13$ & $7$ & $\geq5$ & \ref{Fp_gonality_large}, \ref{magmamapprop}, \ref{towerthmcor}\\
            & & & $\{\pm1,\pm19\}$ & $9$ & $6$ & $geq5$ & \ref{Fp_gonality_deg5}, \ref{quotientmapprop}, \ref{bettiprop}\\
            & & & $\{\pm1,\pm3,\pm9,\pm13\}$ & $7$ & $4$ & $4$ & \ref{quotientmapprop}\\
            & & & $\{\pm1,\pm7,\pm9,\pm17\}$ & $7$ & $4$ & $4$ & \ref{quotientmapprop}\\
            & & & $\{\pm1,\pm9,\pm11,\pm19\}$ & $5$ & $4$ & $4$ & \ref{genus5gonalmap}\\
     
  \hline

\end{longtable}  
\end{center}

\clearpage
\begin{center}
\begin{longtable}{|c|c|c|c|c|c|c|c|}
\caption{The $\Q$-gonalities of curves $X_\Delta(N)$ for $N\geq41$.}
\label{tab:main2}\\
\hline

  $N$ & $(\Z/N\Z)^\times$ & generators & $\Delta$ & $g$ & $\textup{gon}_\Q$ & $\textup{gon}_\C$ & results used\\
  \hline

    $41$ & $C_{40}$ & $6$ & $\{\pm1,\pm9\}$ & $21$ & $\geq6$ & & \ref{poonen}(vii)\\
    & & & $\{\pm1,\pm3,\pm9,\pm14\}$ & $11$ & $\geq6$ & & \ref{Fp_gonality_deg5}, \ref{towerthmcor}\\
    & & & $\{\pm1,\pm4,\pm10,\pm16,\pm18\}$ & $11$ & $\geq6$ & & \ref{Fp_gonality_deg5}, \ref{towerthmcor}\\
    & & & $\left<-1,5\right>$ & $5$ & $4$ & $4$ & \ref{genus5gonalmap}\\

    \hline

    $42$ & $C_2\times C_6$ & $5,13$ & $\{\pm1,\pm13\}$ & $13$ & $\geq6$ & & \ref{Fp_gonality_deg5}\\
    & & & $\{\pm1,\pm5,\pm17\}$ & $9$ & $\geq6$ & $\geq5$ & \ref{Fp_gonality_deg5}, \ref{bettiprop}\\

    \hline

    $43$ & $C_{42}$ & $3$ & $\{\pm1,\pm6,\pm7\}$ & $15$ & $\geq6$ & & \ref{Fp_gonality_deg5}\\
    & & & $\left<-1,27\right>$ & $9$ & $\geq6$ & $\geq5$ & \ref{Fp_gonality_deg5}, \ref{bettiprop}\\
    
    \hline

    $44$ & $C_2\times C_{10}$ & $-1,3$ & $\{\pm1,\pm21\}$ & $16$ & $\geq6$ & & \ref{Fp_gonality_deg5}\\
    & & & $\{\pm1,\pm5,\pm7,\pm9,\pm19\}$ & $8$ & $5$ & $5$ & \ref{magmamapprop}, \ref{bettiprop}\\

    \hline

    $45$ & $C_2\times C_{12}$ & $-1,2$ & $\{\pm1,\pm19\}$ & $21$ & $\geq6$ & & \ref{poonen}(vii)\\
    & & & $\{\pm1,\pm14,\pm16\}$ & $9$ & $\geq6$ & $\geq5$ & \ref{Fp_gonality_deg5}, \ref{bettiprop}\\
    & & & $\{\pm1,\pm8,\pm17,\pm19\}$ & $11$ & $\geq6$ & & \ref{Fp_gonality_deg5}\\
    & & & $\left<-1,4\right>$ & $5$ & $4$ & $4$ & \ref{genus5gonalmap}\\

    \hline

    $48$ & $C_2\times C_2\times C_4$ & $-1,5,7$ & $\{\pm1,\pm7\}$ & $19$ & $\geq6$ & $\geq6$ & \ref{CSprop}\\
    & & & $\{\pm1,\pm17\}$ & $19$ & $\geq6$ & $\geq6$ & \ref{CSprop}\\
    & & & $\{\pm1,\pm23\}$ & $13$ & $\geq6$ & & \ref{Fp_gonality_deg5}\\
    & & & $\{\pm1,\pm5,\pm19,\pm23\}$ & $7$ & $4$ & $4$ & \ref{quotientmapprop}\\
    & & & $\{\pm1,\pm7,\pm17,\pm23\}$ & $7$ & $4$ & $4$ & \ref{quotientmapprop}\\
    & & & $\{\pm1,\pm11,\pm13,\pm23\}$ & $5$ & $4$ & $4$ & \ref{genus5gonalmap}\\

    \hline

    $49$ & $C_{42}$ & $3$ & $\{\pm1,\pm18,\pm19\}$ & $19$ & $\geq6$ & & \ref{Fp_gonality_deg5}\\
    & & & $\left<-1,27\right>$ & $3$ & $3$ & $3$ & $g=3$, \ref{poonen}(iv)\\

    \hline

    $50$ & $C_{20}$ & $3$ & $\{\pm1,\pm7\}$ & $22$ & $\geq6$ & $\geq6$ & \ref{CSprop}\\
    & & & $\{\pm1,\pm9,\pm11,\pm19,\pm21\}$ & $4$ & $3$ & $3$ & \ref{CtrigonalQtetragonalthm}\\

    \hline

    $51$ & $C_2\times C_{16}$ & $-1,3$ & $\{\pm1,\pm16\}$ & $33$ & $\geq6$ & & \ref{poonen}(vii)\\
    & & & $\{\pm1,\pm4,\pm13,\pm16\}$ & $17$ & $\geq6$ & & \ref{poonen}(vii)\\
    & & & $\left<-1,9\right>$ & $9$ & $\geq6$ & $\geq5$ & \ref{Fp_gonality_deg5}, \ref{bettiprop}\\

    \hline

    $52$ & $C_2\times C_{12}$ & $-1,7$ & $\left<-1,7^3\right>$ & $13$ & $\geq6$ & & \ref{Fp_gonality_deg5}\\
    & & & $\left<-1,3\right>$ & $9$ & $\geq6$ & $\geq5$ & \ref{Fp_gonality_deg5}, \ref{bettiprop}\\
    & & & other & & & & \ref{poonen}(vii)\\

    \hline

    $53$ & $C_{52}$ & $2$ & $\{\pm1,\pm23\}$ & $40$ & $\geq6$ & & \ref{indexprop}\\
    & & & $\left<-1,4\right>$ & $8$ & $\geq6$ & $\geq5$ & \ref{Fp_gonality_deg5}, \ref{bettiprop}\\

    \hline

    $54$ & $C_{18}$ & $5$ & $\{\pm1,\pm17,\pm19\}$ & $10$ & $\geq6$ & & \ref{Fp_gonality_deg5}\\

    \hline

    $55$ & $C_2\times C_{20}$ & $2,21$ & $\{\pm1,\pm21\}$ & $41$ & $\geq6$ & & \ref{poonen}(vii)\\
    & & & $\{\pm1,\pm12,\pm21,\pm23\}$ & $21$ & $\geq6$ & & \ref{Fp_gonality_deg5}\\
    & & & $\{\pm1,\pm16,\pm19,\pm24,\pm26\}$ & $17$ & $\geq6$ & & \ref{Fp_gonality_deg5}\\
    & & & $\left<-1,4,21\right>$ & $9$ & $4$ & $4$ & \ref{magmamapprop}\\

    \hline

    $56$ & $C_2\times C_2\times C_6$ & $3,13,29$ & $\{\pm1,\pm13,\pm15,\pm27\}$ & $13$ & $\geq6$ & & \ref{Fp_gonality_deg5}\\
    & & & $\left<-1,9,13\right>$ & $11$ & $\geq6$ & & \ref{Fp_gonality_deg5}\\
    & & & $\left<-1,9,15\right>$ & $11$ & $\geq6$ & & \ref{Fp_gonality_deg5}\\
    & & & $\left<-1,3\right>$ & $9$ & $\geq6$ & $\geq5$ & \ref{Fp_gonality_deg5}, \ref{bettiprop}\\
    & & & other & & & & \ref{poonen}(vii)\\

    \hline

    $57$ & $C_2\times C_{18}$ & $2,20$ & $\left<-1,8,20\right>$ & $13$ & $\geq6$ & & \ref{Fp_gonality_deg5}\\
    & & & $\left<-1,2\right>$ & $9$ & $\geq6$ & $\geq5$ & \ref{Fp_gonality_deg5}, \ref{bettiprop}\\
    & & & other & & & & \ref{poonen}(vii)\\

    \hline

    $58$ & $C_{28}$ & $3$ & $\left<-1,3^7\right>$ & & $\geq6$ & $\geq6$ & \ref{indexprop}\\
    & & & $\left<-1,9\right>$ & $12$ & $\geq6$ & & \ref{Fp_gonality_deg5}\\

    \hline

    $60$ & $C_2\times C_2\times C_4$ & $7,11,19$ & $\{\pm1,\pm7,\pm11,\pm17\}$ & $15$ & $\geq6$ & & \ref{Fp_gonality_deg5}\\
    & & & $\{\pm1,\pm11,\pm13,\pm23\}$ & $15$ & $\geq6$ & & \ref{Fp_gonality_deg5}\\
    & & & $\{\pm1,\pm11,\pm19,\pm29\}$ & $13$ & $\geq6$ & & \ref{Fp_gonality_deg5}\\
    & & & other & & & & \ref{poonen}(vii)\\

    \hline

    $61$ & $C_{60}$ & $2$ & $\left<-1,2^5\right>$ & $16$ & $\geq6$ & & \ref{Fp_gonality_deg5}\\
    & & & $\left<-1,8\right>$ & $12$ & $\geq6$ & & \ref{Fp_gonality_deg5}\\
    & & & $\left<-1,4\right>$ & $8$ & $\geq6$ & $\geq5$ & \ref{Fp_gonality_deg5}, \ref{bettiprop}\\
    & & & other & & & & \ref{poonen}(vii)\\

    \hline

    $62$ & $C_{30}$ & $3$ & $\{\pm1,\pm5,\pm25\}$ & $31$ & $\geq6$ & $\geq6$ & \ref{CSprop}\\
    & & & $\{\pm1,\pm15,\pm23,\pm27,\pm29\}$ & $19$ & $\geq6$ & & \ref{Fp_gonality_deg5}\\

    \hline

    $63$ & $C_6\times C_6$ & $2,5$ & $\left<-1,2\right>$ & $17$ & $\geq6$ & & \ref{Fp_gonality_deg5}\\
    & & & $\left<-1,5\right>$ & $17$ & $\geq6$ & & \ref{Fp_gonality_deg5}\\
    & & & $\left<-1,8,10\right>$ & $17$ & $\geq6$ & & \ref{Fp_gonality_deg5}\\
    & & & $\left<-1,8,20\right>$ & $13$ & $\geq6$ & & \ref{Fp_gonality_deg5}\\
    & & & $\left<-1,4,5\right>$ & $9$ & $\geq6$ & $\geq5$ & \ref{Fp_gonality_deg5}, \ref{bettiprop}\\
    & & & other & & & & \ref{poonen}(vii)\\

    \hline

    $64$ & $C_2\times C_{16}$ & $-1,3$ & $\{\pm1,\pm31\}$ & $37$ & $\geq6$ & & \ref{poonen}(vii)\\
    & & & $\{\pm1,\pm15,\pm17,\pm31\}$ & $13$ & $\geq6$ & & \ref{Fp_gonality_deg5}\\
    & & & $\left<-1,9\right>$ & $5$ & $4$ & $4$ & \ref{genus5gonalmap}\\

    \hline

    $65$ & $C_4\times C_{12}$ & $2,12$ & $\left<-1,8,12\right>$ & $13$ & $\geq6$ & & \ref{Fp_gonality_deg5}\\
    & & & $\left<-1,4,12\right>$ & $11$ & $\geq6$ & & \ref{Fp_gonality_deg5}\\
    & & & $\left<-1,4,24\right>$ & $11$ & $\geq6$ & & \ref{Fp_gonality_deg5}\\
    & & & $\left<-1,2\right>$ & $9$ & $\geq6$ & $\geq5$ & \ref{Fp_gonality_deg5}, \ref{bettiprop}\\
    & & & other & & & & \ref{poonen}(vii)\\

    \hline

    $66$ & $C_2\times C_{10}$ & $-1,5$ & $\left<-1,25\right>$ & $17$ & $\geq6$ & & \ref{Fp_gonality_deg5}\\
    & & & $\left<-1,5^5\right>$ & & $\geq6$ & $\geq6$ & \ref{indexprop}\\

    \hline

    $67$ & $C_{66}$ & $2$ & $\left<-1,2^{11}\right>$ & & $\geq6$ & $\geq6$ & \ref{indexprop}\\
    & & & $\left<-1,8\right>$ & $15$ & $\geq6$ & & \ref{Fp_gonality_deg5}\\

    \hline

    $68$ & $C_2\times C_{16}$ & $-1,3$ & $\left<-1,9\right>$ & $13$ & $\geq6$ & & \ref{Fp_gonality_deg5}\\

    \hline

    $69$ & $C_2\times C_{22}$ & $-1,2$ & $\{\pm1,\pm22\}$ & $67$ & $\geq6$ & $\geq6$ & \ref{indexprop}\\
    & & & $\left<-1,4\right>$ & $13$ & $\geq6$ & & \ref{Fp_gonality_deg5}\\

    \hline

    $70$ & $C_2\times C_{12}$ & $-1,3$ & $\left<-1,27\right>$ & $25$ & $\geq6$ & & \ref{Fp_gonality_deg5}\\
    & & & $\left<-1,9\right>$ & $17$ & $\geq6$ & & \ref{Fp_gonality_deg5}\\
    & & & other & & & & \ref{poonen}(vii)\\

    \hline

    $71$ & $C_{70}$ & $7$ & $\{\pm1,\pm5\pm14,\pm17,\pm25\}$ & $36$ & $\geq6$ & & \ref{Fp2pointscomputation}\\
    & & & $\left<-1,7^5\right>$ & $26$ & $\geq6$ & & \ref{Fp_gonality_deg5}\\

    \hline

    $72$ & $C_2\times C_2\times C_6$ & $-1,5,17$ & $\{\pm1,\pm17,\pm19,\pm35\}$ & $21$ & $\geq6$ & $\geq6$ & \ref{CSprop}\\
    & & & $\left<-1,5\right>$ & $13$ & $\geq6$ & & \ref{Fp_gonality_deg5}\\
    & & & $\left<-1,17,25\right>$ & $13$ & $\geq6$ & & \ref{Fp_gonality_deg5}\\
    & & & $\left<-1,13,25\right>$ & $9$ & $\geq6$ & $\geq5$ & \ref{Fp_gonality_deg5}, \ref{bettiprop}\\
    & & & other & & & & \ref{poonen}(vii)\\

    \hline

    $73$ & $C_{72}$ & $5$ & $\left<-1,5^3\right>$ & $13$ & $\geq6$ & & \ref{Fp_gonality_deg5}\\
    & & & $\left<-1,25\right>$ & $9$ & $\geq6$ & $\geq5$ & \ref{Fp_gonality_deg5}, \ref{bettiprop}\\
    & & & other & & & & \ref{poonen}(vii)\\

    \hline

    $74$ & $C_{36}$ & $5$ & $\left<-1,5^3\right>$ & $22$ & $\geq6$ & $\geq6$ & \ref{CSprop}\\
    & & & $\left<-1,25\right>$ & $16$ & $\geq6$ & & \ref{Fp_gonality_deg5}\\
    & & & other & & & & \ref{poonen}(vii)\\

    \hline

    $75$ & $C_2\times C_{20}$ & $-1,2$ & $\{\pm1,\pm26\}$ & $73$ & $\geq6$ & & \ref{poonen}(vii)\\
    & & & $\{\pm1,\pm7,\pm26,\pm32\}$ & $37$ & $\geq6$ & $\geq6$ & \ref{indexprop}\\
    & & & $\{\pm1,\pm14,\pm16,\pm29,\pm31\}$ & $17$ & $\geq6$ & & \ref{Fp_gonality_deg5}\\
    & & & $\left<-1,4\right>$ & $9$ & $4$ & $4$ & \ref{quotientmapprop}\\

    \hline

    $77$ & $C_2\times C_{30}$ & $-1,2$ & $\left<-1,32\right>$ & $31$ & $\geq6$ & & \ref{Fp_gonality_deg5}\\
    & & & $\left<-1,8\right>$ & $19$ & $\geq6$ & & \ref{Fp_gonality_deg5}\\
    & & & $\left<-1,4\right>$ & $13$ & $\geq6$ & & \ref{Fp_gonality_deg5}\\
    & & & other & & & & \ref{poonen}(vii)\\

    \hline

    $78$ & $C_2\times C_{12}$ & $5,7$ & $\left<-1,5,7^3\right>$ & $31$ & $\geq6$ & & \ref{Fp2pointscomputation}\\
    & & & $\left<-1,35,49\right>$ & $31$ & $\geq6$ & & \ref{Fp_gonality_deg5}\\
    & & & other & & & & \ref{poonen}(vii)\\

    \hline

    $79$ & $C_{78}$ & $3$ & $\{\pm1,\pm23,\pm24\}$ & $66$ & $\geq6$ & $\geq6$ & \ref{indexprop}\\
    & & & $\left<-1,27\right>$ & $18$ & $\geq6$ & & \ref{Fp_gonality_deg5}\\

    \hline

    $80$ & $C_2\times C_4\times C_4$ & $-1,3,7$ & $\left<-1,7,9\right>$ & $15$ & $\geq6$ & & \ref{Fp_gonality_deg5}\\
    & & & $\left<-1,3,49\right>$ & $15$ & $\geq6$ & & \ref{Fp_gonality_deg5}\\
    & & & $\left<-1,21,49\right>$ & $13$ & $\geq6$ & & \ref{Fp_gonality_deg5}\\
    & & & other & & & & \ref{poonen}(vii)\\

    \hline

    $81$ & $C_{54}$ & $2$ & $\{\pm1,\pm26,\pm28\}$ & $46$ & & & \ref{poonen}(vii)\\
    & & & $\left<-1,8\right>$ & $10$ & $\geq6$ & & \ref{Fp_gonality_deg5}\\

    \hline

    $85$ & $C_4\times C_{16}$ & $3,13$ & $\left<-1,3\right>$ & $15$ & $\geq6$ & & \ref{Fp_gonality_deg5}\\
    & & & $\left<-1,9,13\right>$ & $13$ & $\geq6$ & & \ref{Fp_gonality_deg5}\\
    & & & other & & & & \ref{poonen}(vii)\\

    \hline

    $87$ & $C_2\times C_{28}$ & $-1,2$ & $\left<-1,2^7\right>$ & & $\geq6$ & $\geq6$ & \ref{indexprop}\\
    & & & $\left<-1,4\right>$ & $17$ & $\geq6$ & & \ref{Fp_gonality_deg5}\\

    \hline

    $88$ & $C_2\times C_2\times C_{10}$ & $-1,5,21$ & $\left<-1,21,5^5\right>$ & $41$ & $\geq6$ & $\geq6$ & \ref{indexprop}\\
    & & & $\left<-1,21,25\right>$ & $19$ & $\geq6$ & & \ref{Fp2pointscomputation}\\
    & & & $\left<-1,25,105\right>$ & $19$ & $\geq6$ & & \ref{Fp_gonality_deg5}\\
    & & & $\left<-1,5\right>$ & $17$ & $\geq6$ & & \ref{Fp_gonality_deg5}\\
    & & & other & & & & \ref{poonen}(vii)\\

    \hline

    $89$ & $C_{88}$ & $3$ & $\{\pm1,\pm34\}$ & $133$ & $\geq6$ & & \ref{poonen}(vii)\\
    & & & $\{\pm1,\pm12,\pm34,\pm37\}$ & $67$ & $\geq6$ & $\geq6$ & \ref{indexprop}\\
    & & & $\left<-1,3^4\right>$ & $27$ & $\geq6$ & & \ref{poonen}(vii)\\
    & & & $\left<-1,9\right>$ & $13$ & $\geq6$ & & \ref{Fp_gonality_deg5}\\

    \hline

    $91$ & $C_6\times C_{12}$ & $2,12$ & $\left<-1,8,12\right>$ & $23$ & $\geq6$ & & \ref{Fp_gonality_deg5}\\
    & & & $\left<-1,8,48\right>$ & $23$ & $\geq6$ & & \ref{Fp_gonality_deg5}\\
    & & & $\left<-1,8,24\right>$ & $21$ & $\geq6$ & & \ref{Fp_gonality_deg5}\\
    & & & $\left<-1,2\right>$ & $21$ & $\geq6$ & & \ref{Fp2pointscomputation}\\
    & & & $\left<-1,4,12\right>$ & $13$ & $\geq6$ & & \ref{Fp_gonality_deg5}\\
    & & & other & & & & \ref{poonen}(vii)\\

    \hline

    $92$ & $C_2\times C_{22}$ & $-1,3$ & $\{\pm1,\pm45\}$ & $100$ & $\geq6$ & $\geq6$ & \ref{indexprop}\\
    & & & $\left<-1,9\right>$ & $20$ & $\geq6$ & & \ref{Fp_gonality_deg5}\\

    \hline

    $95$ & $C_2\times C_{36}$ & $-1,2$ & $\left<-1,8\right>$ & $25$ & $\geq6$ & & \ref{Fp_gonality_deg5}\\
    & & & $\left<-1,4\right>$ & $17$ & $\geq6$ & & \ref{Fp_gonality_deg5}\\
    & & & other & & & & \ref{poonen}(vii)\\

    \hline

    $96$ & $C_2\times C_2\times C_8$ & $-1,5,17$ & $\left<-1,17,25\right>$ & $17$ & $\geq6$ & & \ref{Fp_gonality_deg5}\\
    & & & $\left<-1,25,85\right>$ & $17$ & $\geq6$ & & \ref{Fp_gonality_deg5}\\
    & & & $\left<-1,5\right>$ & $17$ & $\geq6$ & & \ref{Fp2pointscomputation}\\
    & & & other & & & & \ref{poonen}(vii)\\

    \hline

    $98$ & $C_{42}$ & $3$ & $\left<-1,3^7\right>$ & & $\geq6$ & $\geq6$ & \ref{indexprop}\\
    & & & $\left<-1,27\right>$ & $19$ & $\geq6$ & $\geq6$ & \ref{CSprop}\\

    \hline

    $100$ & $C_2\times C_{20}$ & $-1,3$ & $\left<-1,3^5\right>$ & & $\geq6$ & $\geq6$ & \ref{indexprop}\\
    & & & $\left<-1,9\right>$ & $12$ & $\geq6$ & & \ref{Fp_gonality_deg5}\\

    \hline

    $101$ & $C_{100}$ & $2$ & $\left<-1,32\right>$ & $36$ & $\geq6$ & $\geq6$ & \ref{indexprop}\\
    & & & $\left<-1,4\right>$ & $16$ & $\geq6$ & & \ref{Fp_gonality_deg5}\\
    & & & other & & & & \ref{poonen}(vii)\\

    \hline

    $103$ & $C_{102}$ & $5$ & $\left<5^{17}\right>$ & & $\geq6$ & $\geq6$ & \ref{indexprop}\\
    & & & $\left<-1,5^3\right>$ & $24$ & $\geq6$ & & \ref{Fp_gonality_deg5}\\
    & & & other & & & & \ref{poonen}(vii)\\

    \hline

    $104$ & $C_2\times C_2\times C_{12}$ & $-1,15,51$ & $\left<-1,15^3,51\right>$ & $31$ & $\geq6$ & & \ref{Fp2pointscomputation}\\
    & & & $\left<-1,15\right>$ & $23$ & $\geq6$ & & \ref{Fp_gonality_deg5}\\
    & & & $\left<-1,15^2,15\cdot51\right>$ & $23$ & $\geq6$ & & \ref{Fp_gonality_deg5}\\
    & & & $\left<-1,15^2,51\right>$ & $21$ & $\geq6$ & & \ref{Fp2pointscomputation}\\
    & & & other & & & & \ref{poonen}(vii)\\

    \hline

    $109$ & $C_{108}$ & $6$ & $\left<-1,6^3\right>$ & $22$ & $\geq6$ & & \ref{Fp_gonality_deg5}\\
    & & & $\left<-1,36\right>$ & $16$ & $\geq6$ & & \ref{Fp_gonality_deg5}\\
    & & & other & & & & \ref{poonen}(vii)\\

    \hline

    $111$ & $C_2\times C_{36}$ & $-1,2$ & $\left<-1,8\right>$ & $31$ & $\geq6$ & & \ref{Fp_gonality_deg5}\\
    & & & $\left<-1,4\right>$ & $21$ & $\geq6$ & & \ref{Fp_gonality_deg5}\\
    & & & other & & & & \ref{poonen}(vii)\\

    \hline

    $119$ & $C_2\cdot C_{48}$ & $-1,3$ & $\left<-1,27\right>$ & $31$ & $\geq6$ & & \ref{Fp_gonality_deg5}\\
    & & & $\left<-1,9\right>$ & $21$ & $\geq6$ & & \ref{Fp_gonality_deg5}\\
    & & & other & & & & \ref{poonen}(vii)\\

    \hline

    $121$ & $C_{110}$ & $2$ & $\left<-1,2^{11}\right>$ & & $\geq6$ & $\geq6$ & \ref{indexprop}\\
    & & & $\left<-1,32\right>$ & & $\geq6$ & $\geq6$ & \ref{indexprop}\\

    \hline

    $125$ & $C_{100}$ & $2$ & $\left<-1,32\right>$ & & $\geq6$ & $\geq6$ & \ref{indexprop}\\
    & & & $\left<-1,4\right>$ & $16$ & $5$ & $5$ & \ref{Fp_gonality_deg5}, \ref{quotientmapprop}\\
    & & & other & & & & \ref{poonen}(vii)\\

    \hline

    $131$ & $C_{130}$ & $2$ & all $\Delta$ & & $\geq6$ & $\geq6$ & \ref{indexprop}\\

    \hline

    $142$ & $C_{70}$ & $7$ & all $\Delta$ & & $\geq6$ & $\geq6$ & \ref{indexprop}\\

    \hline

    $143$ & $C_2\times C_{60}$ & $-1,2$ & $\left<-1,32\right>$ & & $\geq6$ & $\geq6$ & \ref{indexprop}\\
    & & & $\left<-1,8\right>$ & $37$ & $\geq6$ & & \ref{Fp2pointscomputation}\\
    & & & $\left<-1,4\right>$ & $25$ & $\geq6$ & & \ref{Fp_gonality_deg5}\\
    & & & other & & & & \ref{poonen}(vii)\\

    \hline

    $191$ & $C_{190}$ & $19$ & all $\Delta$ & & $\geq6$ & $\geq6$ & \ref{indexprop}\\
   \hline

\end{longtable}
\end{center}

\bibliographystyle{siam}
\bibliography{bibliography1}

\end{document}